\documentclass[final]{aomart}
\usepackage{amssymb}
\usepackage{centernot}
\usepackage{mathtools}
\usepackage{enumitem}
\usepackage{comment}
\usepackage{hyperref}
\usepackage{mathtools, nccmath}
\hypersetup{colorlinks=true,linkcolor=blue,filecolor=blue,citecolor = blue,urlcolor=blue}
\copyrightnote{}
\makeatletter
\def\@copyrightyear{\relax}
\makeatother

\newcommand{\N}{\mathbb{N}}
\newcommand{\Z}{\mathbb{Z}}
\newcommand{\col}{\normalfont{\text{Col}}}
\newtheorem{definition}{Definition}
\newtheorem{theorem}{Theorem}

\newtheorem{corollary}{Corollary}

\newtheorem{proposition}{Proposition}
\title{A \textit{p}-adic approach to piecewise polynomial dynamical systems}
\author{Vinny Pagano}
\address{Princeton University, Princeton, New Jersey}
\email{vpagano@princeton.edu}
\thanks{The author would like to thank Dr.\@ Mark McConnell for their encouragement, as well as Professor Jeffrey Lagarias for confirming the originality of this work.}
\begin{document}
\maketitle
\thispagestyle{empty}
\begin{abstract}
\noindent Using \(p\)-adic numbers, we partially categorize the cycles of a sizable class of polynomial dynamical systems. In turn, we prove a few results related to the non-trivial cycles of the \textit{Collatz map} \(\col : \Z_+ \to \Z_+\) defined by $$\col(n) = \begin{cases}
3n + 1, & n \text{ is odd;}
\\ n/2, & \text{otherwise.}\end{cases}$$ Proving the non-existence of non-trivial Collatz cycles would reduce the Collatz conjecture to whether the Collatz map can diverge to infinity.
\end{abstract}
\tableofcontents
\section{Main Results}
\begin{definition}
Given some property \(\mathfrak{P}\) and a cycle \( (\alpha_0,\dots,\alpha_{m-1})\) having \(\alpha_i = \alpha_{m + i}\) for \(i \geq 0\), define the sum and number of terms which satisfy \(\mathfrak{P}\) as:
\[S_{\mathfrak{P}} = \sum_{i = 0 \text{ : } \mathfrak{P}}^{m-1} \alpha_i, \quad N_{\mathfrak{P}} = \sum_{i = 0 \text{ : } \mathfrak{P}}^{m-1} 1.\]
Moreover, set \(S = S_{\mathfrak{P}} + S_{\neg\mathfrak{P}}\).
\end{definition}
\begin{theorem}\label{main}
The integer cycles generated by any function of the form
\[f(x) = \begin{cases}
\sum_{i = 0}^{N_1} a_i x^i/p, & p \mid x;\\
\sum_{i = 0}^{N_2} b_i x^i, & \normalfont{\text{otherwise}}
\end{cases}\]
for \(a_i, b_i \in \Z,\) \(p \in \Z_+\) must satisfy:\[(p - 1) S = a_0 N_{p \mid x} + (a_1 - 1)S_{p \mid x} + \sum_{i = 2}^{N_1} a_i S_{p \mid x}^i + pb_0 N_{p \centernot \mid x} + (pb_1 - 1)S_{p \centernot \mid x} + \sum_{i = 2}^{N_2} pb_i S_{p \centernot \mid x}^i.\]
\end{theorem}

We first establish a wonderful correspondence theorem between orbits (obtained via successive iterations of the map) and \(p\)-adic numbers.

\begin{theorem}\label{p-adic correspondence theorem}
For a given integer $n$, define $\alpha_0 = n$ and $\alpha_i = f(\alpha_{i - 1})$ for $i > 0$. In the subring $\Z_p$ of \(p\)-adic rationals $\mathbb{Q}_p$, we have the following:
\begin{equation*}
\mathclap{\begin{aligned}
-n = \sum_{p \mid \alpha_i}\left(a_0 + (a_1 - 1)\alpha_i + \sum_{k = 2}^{N_1} a_k \alpha_i^k \right)\cdot p^i + \sum_{p \centernot\mid \alpha_i} \left(pb_0 + (pb_1 - 1)\alpha_i + \sum_{k = 2}^{N_2} pb_k \alpha_i^k\right) \cdot p^i.\end{aligned}}
\end{equation*}
\end{theorem}
\begin{proof}
Since $\alpha_i \in \Z$, we can define the \(p\)-adic integer
$$A = \sum_{i = 0}^{\infty}\alpha_i \cdot p^i = n + p\sum_{i = 0}^{\infty}f(\alpha_i) \cdot p^i.$$
We now decompose the expression and expand accordingly:
\begin{equation*}
\mathclap{\begin{aligned}
A &= n + p \sum_{p \mid \alpha_i} f(\alpha_i) \cdot p^i + p \sum_{p \centernot\mid \alpha_i} f(\alpha_i) \cdot p^i = n + p \sum_{p \mid \alpha_i} \sum_{k = 0}^{N_1} \frac{a_k \alpha_i^k}{p}\cdot p^i + p \sum_{p \centernot\mid \alpha_i} \sum_{k = 0}^{N_2} b_k \alpha_i^k \cdot p^i \\
  &= n + \sum_{p \mid \alpha_i} \alpha_i \cdot p^i + \sum_{p \mid \alpha_i}\left(a_0 + (a_1 - 1)\alpha_i + \sum_{k = 2}^{N_1} a_k \alpha_i^k \right)\cdot p^i + \sum_{p \centernot\mid \alpha_i} \alpha_i \cdot p^i + \sum_{p \centernot\mid \alpha_i} \left(pb_0 + (pb_1 - 1)\alpha_i + \sum_{k = 2}^{N_2} pb_k \alpha_i^k\right) \cdot p^i\\
  &= n + A + \sum_{p \mid \alpha_i}\left(a_0 + (a_1 - 1)\alpha_i + \sum_{k = 2}^{N_1} a_k \alpha_i^k \right)\cdot p^i + \sum_{p \centernot\mid \alpha_i} \left(pb_0 + (pb_1 - 1)\alpha_i + \sum_{k = 2}^{N_2} pb_k \alpha_i^k\right) \cdot p^i.
\end{aligned}}
\end{equation*}
This completes the proof. \end{proof}
Remarkably, the previous formula has a representation in $\Z$ for cycles.
\begin{theorem}\label{Z_p-to-Z theorem}
Suppose we have a cycle $K = (\alpha_0,\dots, \alpha_{m-1})$ so that $\alpha_i = \alpha_{m + i}$ for all $i \ge 0$. Then, for all $j \in \{0,\dots,m-1\}$,

\begin{equation*}\mathclap{\begin{aligned}\alpha_j \cdot (p^m - 1) = \sum_{i = 0 \text{ : } p \mid \alpha_{i + j}}^{m-1}\left(a_0 + (a_1 - 1)\alpha_{i + j} + \sum_{k = 2}^{N_1} a_k \alpha_{i + j}^k \right)\cdot p^i + \sum_{i = 0 \text{ : } p \centernot\mid \alpha_{i + j}}^{m - 1} \left(pb_0 + (pb_1 - 1)\alpha_{i + j} + \sum_{k = 2}^{N_2} pb_k \alpha_{i + j}^k\right) \cdot p^i.
\end{aligned}}\end{equation*}

\end{theorem}
\begin{proof}
By \autoref{p-adic correspondence theorem} and exploiting the cyclic nature of the orbit,
\begin{equation*}
\mathclap{\begin{aligned}
-\alpha_j &= \sum_{p \mid \alpha_{i + j}}\left(a_0 + (a_1 - 1)\alpha_{i + j} + \sum_{k = 2}^{N_1} a_k \alpha_{i + j}^k \right)\cdot p^i + \sum_{p \centernot\mid \alpha_{i + j}} \left(pb_0 + (pb_1 - 1)\alpha_{i + j} + \sum_{k = 2}^{N_2} pb_k \alpha_{i + j}^k\right) \cdot p^i\\
&= \left(\sum_{i = 0 \text{ : } p \mid \alpha_{i + j}}^{m-1}\left(a_0 + (a_1 - 1)\alpha_{i + j} + \sum_{k = 2}^{N_1} a_k \alpha_{i + j}^k \right)\cdot p^i + \sum_{i = 0 \text{ : } p \centernot\mid \alpha_{i + j}}^{m - 1} \left(pb_0 + (pb_1 - 1)\alpha_{i + j} + \sum_{k = 2}^{N_2} pb_k \alpha_{i + j}^k\right) \cdot p^i\right) \cdot \sum_{\ell = 0}^{\infty} p^{\ell \cdot m}\\
&= \left(\sum_{i = 0 \text{ : } p \mid \alpha_{i + j}}^{m-1}\left(a_0 + (a_1 - 1)\alpha_{i + j} + \sum_{k = 2}^{N_1} a_k \alpha_{i + j}^k \right)\cdot p^i + \sum_{i = 0 \text{ : } p \centernot\mid \alpha_{i + j}}^{m - 1} \left(pb_0 + (pb_1 - 1)\alpha_{i + j} + \sum_{k = 2}^{N_2} pb_k \alpha_{i + j}^k\right) \cdot p^i\right) \cdot \frac{1}{1 - p^m}.
\end{aligned}}
\end{equation*}
Since the last step is a property enjoyed by \(p\)-adic expansions, we are done.
\end{proof}
We can now relate the sum of the terms either divisible by \(p\) or not with the sum of all terms in a cycle. This is explicated through the next theorem.
\begin{theorem}\label{podd theorem}
Suppose we have a cycle $K = (\alpha_0,\dots, \alpha_{m-1})$. Then,
\begin{equation*}\mathclap{
\begin{aligned}
(p - 1) \sum_{i = 0}^{m - 1} \alpha_i = \sum_{i = 0 \text{ : } p \mid \alpha_{i}}^{m-1}\left(a_0 + (a_1 - 1)\alpha_{i} + \sum_{k = 2}^{N_1} a_k \alpha_{i}^k \right) + \sum_{i = 0 \text{ : } p \centernot\mid \alpha_{i}}^{m - 1} \left(pb_0 + (pb_1 - 1)\alpha_{i} + \sum_{k = 2}^{N_2} pb_k \alpha_{i}^k\right).
\end{aligned}}
\end{equation*}
\end{theorem}

\begin{proof}
We sum over index \(j\) both sides of the equation of \autoref{Z_p-to-Z theorem}, noting that the cyclic symmetry of \(i\) and \(j\) on the sum allows us to relabel the indices without having to relabel the \(p^i\) terms. As a result, we can extract each \(p^i\) term from the invariant sum:

\begin{equation*}\mathclap{
\begin{aligned}
\alpha_j \cdot (p^m - 1) &= \sum_{i = 0 \text{ : } p \mid \alpha_{i + j}}^{m-1}\left(a_0 + (a_1 - 1)\alpha_{i + j} + \sum_{k = 2}^{N_1} a_k \alpha_{i + j}^k \right)\cdot p^i + \sum_{i = 0 \text{ : } p \centernot\mid \alpha_{i + j}}^{m - 1} \left(pb_0 + (pb_1 - 1)\alpha_{i + j} + \sum_{k = 2}^{N_2} pb_k \alpha_{i + j}^k\right) \cdot p^i\\
(p^m - 1) \sum_{j = 0}^{m-1} \alpha_j &= \sum_{i = 0}^{m-1} p^i \left(\sum_{j = 0 \text{ : } p \mid \alpha_{i + j}}^{m-1}\left(a_0 + (a_1 - 1)\alpha_{i + j} + \sum_{k = 2}^{N_1} a_k \alpha_{i + j}^k \right) + \sum_{j = 0 \text{ : } p \centernot\mid \alpha_{i + j}}^{m - 1} \left(pb_0 + (pb_1 - 1)\alpha_{i + j} + \sum_{k = 2}^{N_2} pb_k \alpha_{i + j}^k\right)\right)\\
&= \frac{p^m - 1}{p - 1} \left(\sum_{j = 0 \text{ : } p \mid \alpha_{j}}^{m-1}\left(a_0 + (a_1 - 1)\alpha_{j} + \sum_{k = 2}^{N_1} a_k \alpha_{j}^k \right) + \sum_{j = 0 \text{ : } p \centernot\mid \alpha_{j}}^{m - 1} \left(pb_0 + (pb_1 - 1)\alpha_{j} + \sum_{k = 2}^{N_2} pb_k \alpha_{j}^k\right)\right)
\end{aligned}}
\end{equation*}
Relabeling once more, we obtain the formula
\begin{equation*}\mathclap{
\begin{aligned}
(p - 1) \sum_{i = 0}^{m - 1} \alpha_i = \sum_{i = 0 \text{ : } p \mid \alpha_{i}}^{m-1}\left(a_0 + (a_1 - 1)\alpha_{i} + \sum_{k = 2}^{N_1} a_k \alpha_{i}^k \right) + \sum_{i = 0 \text{ : } p \centernot\mid \alpha_{i}}^{m - 1} \left(pb_0 + (pb_1 - 1)\alpha_{i} + \sum_{k = 2}^{N_2} pb_k \alpha_{i}^k\right),
\end{aligned}}
\end{equation*}
as desired.
\end{proof}
Using the terminology from the beginning of the paper gives us \autoref{main}.
\section{Some Consequences}
The previous theorems reduce the question of cyclic overlap among piecewise polynomial dynamical systems to solving Diophantine equations. It is perhaps helpful to restate these theorems for the Collatz map.
\begin{corollary}\label{2-adic correspondence theorem}
For any $n \in \N$, define $\alpha_0 = n$ and $\alpha_i = \col(\alpha_{i - 1})$ for $i > 0$. In the subring $\Z_2$ of dyadic rationals $\mathbb{Q}_2$, we have the following:
\[-n = \sum_{\alpha_i \normalfont{\text{ odd}}}(5 \alpha_i + 2) \cdot 2^i.\]
\end{corollary}
\begin{corollary}\label{Z_2-to-Z theorem}
Suppose we have a $\col$-cycle $K = (\alpha_0,\dots, \alpha_{m-1})$ so that $\alpha_i = \alpha_{m + i}$ for all $i \ge 0$. Then, for all $j \in \{0,\dots,m-1\}$,
\begin{align*}
\alpha_j \cdot (2^m - 1) &= \sum_{i = 0 \text{ : } \alpha_{i + j} \normalfont{\text{{ odd}}}}^{m - 1} (5 \alpha_{i + j} + 2) \cdot 2^i.
\end{align*}
\end{corollary}
\begin{corollary}\label{odd theorem}
Suppose we have a $\col$-cycle $K = (\alpha_0,\dots, \alpha_{m-1})$. Then,
\begin{align*}
\sum_{i = 0}^{m-1} \alpha_i &= \sum_{i = 0 \text{ : } \alpha_i \normalfont{\text{ odd}}}^{m-1} (5 \alpha_i + 2).
\end{align*}
\end{corollary}
By substitution into \fullref{Corollary}{odd theorem}, we are left with an elegant equation:
\begin{equation}\label{eq}
S = 5S_{odd} + 2N_{odd}.
\end{equation}
To the author's knowledge, while \fullref{Equation}{eq} does not appear to be present in the literature, it does not lend itself to any computational efficiencies with respect to the search for non-trivial cycles. However, we can apply the same technique to the \textit{inverse Collatz map} \(\col^{-1}\). Observe that, until we prove something about the number of non-trivial cycles, the inverse Collatz map is not well-defined. For example, if there are no other non-trivial cycles\footnote{This gives us an upper (resp. lower) bound on the number of even (resp. odd) terms in an inverse Collatz cycle.}, the inverse Collatz map would be the following:
\[\col^{-1}(n) = \begin{cases}
(n-1)/3, & n \text{ even, } 3 \mid (n - 1);\\
2n, & \text{otherwise}.
\end{cases}\]
Nevertheless, clearly there is one that exists, and we can say something about how it manifests itself using triadic numbers via \autoref{main}.
\begin{theorem}
Any candidate \(\col^{-1}\)-cycle must satisfy
\[2S \leq 5S_{odd'} - N_{even'}.\]
\end{theorem}

Just as \(\col\) appears to almost always converge to 1, \(\col^{-1}\) appears to almost always diverge to infinity. And indeed the right notion of the inverse function will coincide with \(\col\) on cycles. So coupled with the observations that \(N_{even'} \leq N_{odd}\) and \(N_{odd'} \geq N_{even}\), we see that 
\[S_{even'} \leq 3S_{odd} + 2N_{odd} - N_{even'}, \quad S_{odd'} \geq 2S_{odd} + N_{even'}\]
\[\implies S_{even'} - S_{odd'} \leq S_{odd} + 2N_{odd} - 2N_{even'}.\]
Surprisingly, there is an even deeper result that is related to \fullref{Equation}{eq}. 
\begin{theorem}
If the Collatz conjecture is true, then for all positive \(n\)
\[S - (5S_{odd} + 2N_{odd}) = 2 (n - 1),\]
where \(S,\) \(S_{odd},\) and \(N_{odd}\) are computed over the orbit \(n, \col(n), \dots, 2\).
\end{theorem}
\begin{proof}
We prove the statement by induction on the length \(\ell\) of the orbit. The case at \(\ell = 1\) is trivial, as \(S = n = 2\) and \(S_{odd} = N_{odd} = 0\). Suppose that for all \(\ell \leq k\) the equation is satisfied, denoting the starting term and respective statistics of the orbit of length \(k + 1\) by \(n'\), \(S'\), \(S_{odd}'\), and \(N_{odd}'\). By assumption, we know that \(\col(n')\) is the start of an orbit which satisfies the inductive hypothesis, so
\[S - (5S_{odd} + 2N_{odd}) = 2(\col(n') - 1),\] where \(S, S_{odd},\) and \(N_{odd}\) are the statistics of the orbit starting at \(\col(n')\).

Suppose that \(n'\) is odd. Then, \(\col(n') = 3n' + 1\) is even, and \(S = S' - n',\) \(S_{odd} = S_{odd}' - n'\), and \(N_{odd} = N_{odd}' - 1\). Thus, 
\[S - (5S_{odd} + 2N_{odd}) = (S' - n') - \left(5(S_{odd}' - n') + 2(N_{odd}' - 1)\right) = 2(\col(n') - 1)\]
\[\implies S' - (5S'_{odd} + 2N'_{odd}) = 2(n' - 1).\]
The proof where \(n'\) is even follows a similar argument.
\end{proof}
Preliminary data suggests that this is a truly unique property of the Collatz map, and that the cycle analogue formula for alternative piecewise polynomials does not give rise to a similar relation.


\begin{thebibliography}{9}
\bibitem{latexcompanion} 
Gouvea, Fernando.``$p$-adic Numbers: An Introduction." 2nd ed., \textit{Springer-Verlag Berlin Heidelberg}, 1997, \url{www.springer.com/gp/book/9783642590580}. 
\bibitem{laxexcompanion}
Lagarias, Jeffrey. ``The Set of Rational Cycles for the $3x + 1$ Problem." \textit{Acta Arithmetica}, vol. 56, no. 1, 1990, pp. 33-53., doi:10.4064/aa-56-1-33-53. 
\bibitem{latexcompanion} 
Lagarias, Jeffrey.
``The Ultimate
Challenge: The $3x + 1$ Problem." \textit{American Mathematical Society}, Providence, RI 2010.

\end{thebibliography}
\end{document}